\documentclass[12pt]{amsart}
\usepackage{amssymb,amsmath,amsthm}
\numberwithin{equation}{section}
\begin{document}

\theoremstyle{plain}
\newtheorem{theorem}{Theorem}[section]
\newtheorem{lemma}[theorem]{Lemma}
\newtheorem{proposition}[theorem]{Proposition}
\newtheorem{corollary}[theorem]{Corollary}
\newtheorem{conjecture}[theorem]{Conjecture}

\theoremstyle{definition}
\newtheorem*{definition}{Definition}

\theoremstyle{remark}
\newtheorem*{remark}{Remark}
\newtheorem{example}{Example}[section]
\newtheorem*{remarks}{Remarks}

\newcommand{\cc}{{\mathbf C}}
\newcommand{\qq}{{\mathbf Q}}
\newcommand{\rr}{{\mathbf R}}
\newcommand{\nn}{{\mathbf N}}
\newcommand{\zz}{{\mathbf Z}}
\newcommand{\pp}{{\mathbf P}}
\newcommand{\al}{\alpha}
\newcommand{\be}{\beta}
\newcommand{\ga}{\gamma}
\newcommand{\ze}{\zeta}
\newcommand{\om}{\omega}
\newcommand{\ep}{\epsilon}
\newcommand{\la}{\lambda}
\newcommand{\de}{\delta}
\newcommand{\De}{\Delta}
\newcommand{\Ga}{\Gamma}
\newcommand{\si}{\sigma}

\title{Congruence properties of binary partition functions}

\author{Katherine Anders}
\address{Department of Mathematics, University of 
Illinois at Urbana-Champaign, Urbana, IL 61801} 
\email{kaanders@math.uiuc.edu}

\author{Melissa Dennison}
\address{Department of Math and Computer Science, Baldwin-Wallace College,
Berea, Ohio 44017} 
\email{mdenniso@bw.edu}

\author{Bruce Reznick}
\address{Department of Mathematics, University of 
Illinois at Urbana-Champaign, Urbana, IL 61801} 
\email{reznick@math.uiuc.edu}

\author{Jennifer Weber}
\address{Department of Mathematics, University of 
Illinois at Urbana-Champaign, Urbana, IL 61801} 
\email{jlweber@illinois.edu}

\thanks{The first and fourth authors received support
from National Science Foundation grant DMS 0838434
”EMSW21MCTP: Research Experience for Graduate Students”. }

\begin{abstract}
Let $\mathcal A$ be a finite subset of $\mathbb N$ containing 0, and let
$f(n)$ denote the number of ways to write $n$  in the form $\sum
\ep_j2^j$, where $\ep_j \in \mathcal A$. We show that there exists a
computable $T = 
T(\mathcal A)$ so that the sequence $(f(n)$ mod 2) is periodic with period
$T$. Variations and generalizations of this problem are also discussed. 
\end{abstract}
\date{\today}
\maketitle

\maketitle

\section{Introduction}

 Let $\mathcal A = \{0=a_0 < a_1 <\dots \}$ denote a finite or
 infinite  subset of $\mathbb
N$ containing 0, and fix an integer $b \ge 2$. Let $f_{\mathcal A,
  b}(n)$ denote the number of ways to write $n$ in the form
\begin{equation}\label{E:basicrep}
n = \sum_{k=0}^\infty \ep_k b^k, \qquad \ep_k \in \mathcal A.
\end{equation}
The uniqueness of the standard base-$b$ representation of $n \ge 0$
reflects the fact that 
$f_{\mathcal A, b}(n) = 1$ 
for $\mathcal A = \{0,\dots,b-1\}$. For non-standard
bases, the behavior of $f_{\mathcal A,b}(n)$ has been studied
primarily when $\mathcal A = 
\mathbb N$ or $b=2$, in terms of congruences at special values, and
also asymptotically. 
 In this paper, we are concerned with the behavior of 
$f_{\mathcal A,b}(n) \pmod{d}$, especially when $b=d=2$, and when
$\mathcal A$ is finite.

We associate to $\mathcal A$ its characteristic function
$\chi_{\mathcal A}(n)$, and the generating function
\begin{equation}\label{E:phi}
\phi_{\mathcal A}(x) := \sum_{n=0}^\infty \chi_{\mathcal A}(n)x^n =
\sum_{a \in  \mathcal A} x^a = 1 + x^{a_1} + \cdots. 
\end{equation}
Let
\begin{equation}\label{E:genfun}
F_{\mathcal A, b}(x): = \sum_{n=0}^\infty f_{\mathcal A,b}(n) x^n
\end{equation}
denote the generating function of $f_{\mathcal A, b}(n)$.
Viewing \eqref{E:basicrep} as a partition problem, we find an
immediate  infinite product representation for $F_{\mathcal A, b}(x)$:
\begin{equation}\label{E:infprod}
F_{\mathcal A, b}(x) = \prod_{k=0}^\infty \left(1 + x^{a_1b^k} +
  \cdots \right) = \prod_{k=0}^\infty \phi_{\mathcal A}(x^{b^k}).
\end{equation}

Observe that  \eqref{E:basicrep} implies that $n \equiv \ep_0
\pmod{b}$. Thus, every such representation may be rewritten as 
\begin{equation}\label{E:shift}
n =  \sum_{j=0}^\infty \ep_j b^j = \ep_0 + b\left(\sum_{j=0}^{\infty}
  \ep_{j+1}b^j\right).
\end{equation}
Since  $f_{\mathcal A, b}(n) = 0$
for $n < 0$, we see that \eqref{E:shift}  gives the recurrence
\begin{equation}\label{E:recur}
f_{\mathcal A, b}(n) = \sum_{\substack{a\in \mathcal A,\\ n \equiv a\text{ mod } b }}
f_{\mathcal A, b}\left(\tfrac{n-a}b\right), \quad \text{for} \quad n \ge 1.
\end{equation}
Alternatively, decompose $\mathcal A$ into residue classes mod $b$ and
write
\begin{equation}\label{E:slice}
\mathcal A = \bigcup_{i=0}^{b-1} \mathcal A_i, \quad \text{where} \quad
\mathcal A_i := \mathcal A \cap (b\mathbb Z +i).
\end{equation}
If we write $\mathcal A_i = \{ bv_{k,i} + i\}$, then for $m \ge 0$ and $0 \le
i \le b-1$: 
\begin{equation}\label{E:recur2}
f_{\mathcal A, b}(bm + i) = \sum_k f_{\mathcal A, b}(m-v_{k,i}).
\end{equation}
The initial condition $f_{\mathcal A,b}(0) = 1$, combined with
\eqref{E:recur} or \eqref{E:recur2}, is
sufficient to determine $f_{\mathcal A,b}(n)$ for all $n > 0$.

We say that a sequence $(u_n)$ is {\it eventually periodic} if there
exist integers $N \ge 0$, $T \ge 1$ so that, for $n \ge N$, $u_{n+T}=
u_n$. The {\it period} of an eventually periodic sequence is the
smallest such $T$. By extension, we say that the set $\mathcal A$ is {\it
  eventually periodic} if the sequence of its characteristic function,
$(\chi_{\mathcal A}(n))$, is
eventually periodic. Equivalently, $\mathcal A$ is eventually periodic
if there exists $T$, and integers $r_1,\dots,r_k$, $k \ge 0$, $0 \le
r_i \le T-1$, so that the symmetric set
difference of $\mathcal A$ and $\cup(T \mathbb N + r_i)$ is finite. In
particular, if  $\mathcal A$ is finite or the complement of a finite
set, then $\mathcal A$ is eventually periodic.

The principal result of this paper is the relationship of
$F_{\mathcal A,2}(x)$ and $\phi_{\mathcal A}(x)$ mod 2.
\begin{theorem}\label{MainTheorem}
As elements of  $\mathbb F_2[[x]]$,  
\begin{equation}\label{E:main}
F_{\mathcal A, 2}(x) \phi_{\mathcal A}(x) = 1.  
\end{equation} 
\end{theorem}
Theorem \ref{MainTheorem} has an immediate corollary.
\begin{corollary}\label{MainCor}
\
\begin{enumerate}
\item If $\mathcal A$ is finite, then there is a computable
  integer $T = T(\mathcal A)>0$ so   that for all $n \ge 0$,  
$f_{\mathcal A,2}(n) \equiv f_{\mathcal A,2}(n+T)\pmod{2}$.
\item If $\mathcal A$ is infinite, the sequence $(f_{\mathcal A,2}(n)
  \pmod{2})$ is  eventually periodic if and only if $\phi_{\mathcal A}(x)$ is the
  power series of a rational function in $\mathbb F_2(x)$ if and only
  if  the set $\mathcal A$ is   eventually periodic.
\end{enumerate}
\end{corollary}
It will follow from Corollary \ref{MainCor}(1) that if $\mathcal A$ is a
finite set, and $T =  T(\mathcal A)$, then there is a {\it
  complementary} finite set $\mathcal A' = \{0=b_0  < b_1< \dots\}$ so that 
\begin{equation}\label{E:compleodd}
\begin{gathered}
f_{\mathcal A,2}(n)  \quad \text{is odd} \iff n \equiv b_k \pmod{T}
\quad \text{for some $b_k$}; \\
  f_{\mathcal A',2}(n)  \quad \text{is odd} \iff n \equiv a_k \pmod{T}
\quad \text{for some $a_k$}.
\end{gathered}
\end{equation}
Complementary sets needn't look very much alike.
If $\mathcal A = \{0,1,4,9\}$, then $T = 84$ and $|\mathcal A'|
= 41$, with elements ranging from 0 to 75 (see Example \ref{0149}).

One instance of Theorem \ref{MainTheorem} in the literature comes from 
the {\it Stern sequence} $(s(n))$ (see \cite{Ste, L, Re2}), which is defined by 
\begin{equation}\label{E:ste}
s(0) = 0, \ s(1) = 1; s(2n) = s(n),\ s(2n+1) = s(n) + s(n+1) \quad
\text{for $n \ge 1$}.   
\end{equation}
It was proved in \cite{Re1} that $s(n) = f_{\{0,1,2\},2}(n-1)$, under
which the recurrence \eqref{E:ste} is a  translation of 
\eqref{E:recur2}. It is easy to prove, and has basically been known since
\cite[p.197]{Ste}, that  $s(n)$ is even if and only if $n$ is a multiple of 
three. A simple application of Theorem \ref{MainTheorem} shows that in
$\mathbb F_2(x)$,
\begin{equation}
F_{\{0,1,2\},2}(x) = \frac 1{1+x+x^2} = \frac{1+x}{1+x^3} =
1+x+x^3+x^4+x^6+x^7 + \dots.
\end{equation}
This result was generalized in \cite[Th.2.14]{Re1}, using the infinite
product \eqref{E:infprod}. Here, let $\mathcal A_d =  \{0,\dots,d-1\}$. Then
$\phi_{\mathcal A_d}(x) = \frac{1-x^d}{1-x}$, so in $\mathbb F_2(x)$, 
 \begin{equation}\label{E:d}
F_{\mathcal A_d,2}(x) =  \frac{1+x}{1+x^d} =
1+x+x^d+x^{d+1}+x^{2d}+x^{2d+1} + \dots.
\end{equation}
Thus, $f_{\mathcal A_d,2}(n)$ is odd if and only if $n \equiv 0,1
\pmod{d}$. 
 
We also show that there is no obvious ``universal'' generalization of Theorem
\ref{MainTheorem} to $f_{\mathcal A,b}(n) \pmod{d}$ when $(b,d) \neq (2,2)$.
\begin{theorem}\label{neg}
\
\begin{enumerate}
\item If $(f_{\{0,1,2\},2}(n) \pmod{d})$ is eventually periodic with
  period $T$, then  $(d,T) = (2,3)$.
\item If $d \ge 2$ and $b \ge 3$, then $(f_{\{0,1\},b}(n) \pmod{d})$ is never
  eventually  periodic. 
\end{enumerate}
\end{theorem}
Thus, the Stern sequence has no periodicities mod $d \ge 3$ and, there
exists $\mathcal A$ whose representations in any base $b \ge 3$ have
no periodicity modulo any $d \ge 2$.   

Let $\nu_2(m)$ denote the largest power of 2
dividing $m$. In 1969, Churchhouse \cite{C} conjectured, based on
numerical evidence,  that $f_{\mathbb N,2}(n)$ is
even for $n \ge 2$, that $4 \ | \ f_{\mathbb N,2}(n)$ if and only if either
$\nu_2(n-1)$ or $\nu_2(n)$ is a positive even integer, and that 8 never
divides  $f_{\mathbb N,2}(n)$. He also conjectured that, for all even
$m$, 
\begin{equation}\label{E:Church}
\nu_2(f_{\mathbb N,2}(4m)) - \nu_2(f_{\mathbb N,2}(m)) = \lfloor
\tfrac 32 \dot(3 \nu_2(m) + 4) \rfloor.
\end{equation}
This conjecture was proved in the next few years by R{\o}dseth, and by
Gupta and generalized  
by Hirschhorn and Loxton, R{\o}dseth, Gupta, Andrews, Gupta and
Pleasants, and most recently by  R{\o}dseth and Sellers \cite{RS}.
 We refer the reader to \cite{Re1,RS} for detailed references. 
The statements in Theorem \ref{neg}  about the non-existence of
recurrences do not apply to formulas such as \eqref{E:Church}. On the
other hand, $\phi_{\mathbb N,2}(x) = (1+x)^{-1}$, so Theorem
\ref{MainTheorem} implies that $f_{\mathbb N,2}(n)$ is
even for $n \ge 2$.

The paper is organized as follows. In section two, we review some
basic facts about polynomials and rational functions over $\mathbb
F_2$. In section three, we give two proofs of Theorem  
\ref{MainTheorem} and then prove Corollary \ref{MainCor}. 
In section four, we present several examples and
applications of Theorem \ref{MainTheorem}, as well as a proof of
Theorem \ref{neg}.

Portions of the research in this paper were contained in the Ph.D. dissertation
\cite{D} of the second author, written under the supervision of the third
author, and in the UIUC Summer 2010 Research Experiences for Graduate Students
(REGS) project \cite{AW}  of the first and fourth authors,
written under the supervision of the third author. 

The authors thank Bob McEliece for helpful correspondence.  

\section{Background}

There is an important relationship between rational functions in
$\mathbb F_2[[x]]$ and eventually periodic sequences. 
We first  recall some familiar facts about finite fields, identifying
$\mathbb Z/p\mathbb Z$ with $\mathbb F_p$ for prime $p$. The binomial 
theorem implies that for $a, b \in \mathbb F_p$, $(a+b)^p = a^p + b^p$,
hence $(\sum  a_i)^p = \sum a_i^p$. It follows from this fact and
Fermat's Little Theorem that for any polynomial $f \in \mathbb F_p[x]$,
\begin{equation}\label{E:froshdream}
f(x) = \sum_{j=0}^m a_j x^j \implies f(x)^p = f(x^p).
\end{equation}

If $f \in \mathbb F_2[x]$ is an irreducible polynomial of degree $d$
(so $f(0) \neq 0$), then it is well-known that $f(x) \ | \ 1 +
x^{2^d-1}$. Repeated application of \eqref{E:froshdream} for $p=2$
shows that $(1+x^M)^{2^k} = 1 + x^{2^k\cdot M}$, hence if $f$ is irreducible
and $j \le 2^k$, then $f(x)^j \ | \ 1+x^{2^k\cdot (2^d-1)}$. 
This leads immediately to the following lemma  (see
\cite[Thm.6.21]{B}):  

\begin{lemma}\label{divide}
Suppose $h \in \mathbb F_2[x]$, $h(0) \neq 0$ and $h$ can be factored
over $\mathbb F_2[x]$ as  
\begin{equation}
h = \prod_{i=1}^s f_i^{e_i},
\end{equation}
where the $f_i$ are distinct irreducible polynomials with $\deg(f_i) =
d_i$, and suppose  $2^k \ge e_i$ for all $i$ and some $k \in \mathbb N$. Then 
\begin{equation}\label{E:upbd}
h(x) \ | \ 1 + x^M, \text{ where}\quad M := M(h) = 2^k \cdot
\text{lcm}(2^{d_1}-1,\dots, 2^{d_s}-1). 
\end{equation}
\end{lemma}

Suppose $h \in \mathbb F_2[x]$ and $h(0) = 1$.  The {\it period} of $h$ is the
smallest $T \ge 1$ so that $h(x) \ | \ 1+x^T$; this definition does
not assume that $h$ is irreducible. The period of $h$ can
be much smaller than $M(h)$, however it is always a divisor of $M(h)$.
\begin{lemma}\label{ideal}
If $h$ has period $T$, then $h(x) \ | \ 1+x^V$ in $\mathbb F_2[x]$ 
if and only if $T \ | \ V$.
\end{lemma}
\begin{proof}
We first note that $(1+x^T) \ |\ (1 + x^{kT})$, proving one
direction. For the other, suppose $h(x) \ | \ 1+x^V$; then $V \ge T$.  
Write $V = kT + r$, where $0 \le r \le T-1$. Then
$h(x)$ also divides
\begin{equation}
x^r(1+x^{kT}) + 1 + x^V = 1 + x^r,
\end{equation}
which violates the minimality of $T$ unless $r=0$.
\end{proof}

If $h \in \mathbb F_2[x]$ is irreducible, $\deg h = r$ and the
period of $h$ is $2^r-1$, then  $h$ is called {\it primitive}.
Primitive trinomials have attracted much recent interest,
especially when $2^r-1$ is a Mersenne prime (see \cite{BZ});  Lemma
\ref{divide} implies that all such irreducible $h$ are primitive.
In coding theory, $h$ is called the {\it generator} polynomial and
\begin{equation}
q(x)  = \frac{1 + x^T}{h(x)}
\end{equation} 
is called the {\it parity-check} polynomial.

Consider a rational function in $\mathbb F_2(x)$: 
\begin{equation}\label{E:Euc}
 \frac {g(x)}{h(x)} = a(x) + \frac{r(x)}{h(x)},
\end{equation}
 where $g,h,a,r$ are polynomials, and $\deg r < \deg h$.  We make the
 additional assumption that $h(0) \neq 0$.
Lemma \ref{divide} leads to an important relationship
between rational functions and eventually periodicity.
 
\begin{lemma}\label{periodicity}
Suppose  $b(x) = \sum b_nx^n \in \mathbb F_2[[x]]$ with $b_0 =
1$. Then $b(x)$ is a 
rational function if and only if  $\{n:b_n = 1\}$ is eventually periodic.
\end{lemma}
\begin{proof}
First suppose there exists $T, N$ so that $b_n = b_{n+T}$ for $n \ge N$. 
Then the coefficient of  $x^{n+T}$ in 
\begin{equation}\label{E:rec2}
(1+x^T)\left(\sum_{n=0}^\infty b_n x^n\right)
\end{equation}
is $b_{n+T} + b_{n}=0$ for $n \ge N$. Hence, $b(x)$ is
the quotient of a polynomial of degree $< N$ and $1 + x^T$, and 
is thereby
a rational function. Conversely, suppose $b = \frac gh$ is rational  and
is given by \eqref{E:Euc} with $h(0) = 1$.  Then by Lemma
\ref{divide}, there exists $q(x) \in \mathbb F_2[x]$ and $T$ so that 
 \begin{equation}
b(x) = a(x) + \frac{r(x)}{h(x)} = a(x) + \frac{r(x)q(x)}{1 + x^T},
\end{equation}
hence $(1+x^T)b(x)$ is a polynomial of degree $< N$ (say), so $b_n =
b_{n+T}$ for $n \ge N$. 
\end{proof}

\section{Proofs}
We start this section with  two proofs of Theorem \ref{MainTheorem}. The
first one is somewhat longer, but yields a recurrence of independent
interest. 

As in \eqref{E:slice}, write
\begin{equation}
\begin{gathered}
\mathcal A = \{0=a_0  <a_1 < \cdots \} = \mathcal A_0 \cup \mathcal A_1; \\
\mathcal A_0 = \{0 = 2b_0 < 2b_1 < \cdots \}, \qquad \mathcal A_1
= \{2c_1 + 1< \cdots \}.  
\end{gathered}
\end{equation}
We will write $f_{\mathcal A, 2}(n)$ as $f(n)$ when
there is no ambiguity.
By \eqref{E:recur2}, we have:
\begin{equation}\label{E:bc}
f(2n) = \sum_i f(n-b_i), \qquad f(2n+1) = \sum_j f(n-c_j).
\end{equation}

\begin{theorem}\label{rec}
For all $n \in \mathbb Z$, $n \neq 0$,
\begin{equation}
\Theta(n):= \sum_k f(n - a_k) \equiv 0 \pmod{2}.
\end{equation}  
\end{theorem}
\begin{proof}
If $n < 0$, then $f(n) = 0$, so this is immediate; also  $\Theta(0) = f(0)
= 1$. Suppose  $n > 0$. We distinguish two cases:
$n=2m$ and $n=2m+1$, and  put \eqref{E:bc} back into itself,
diagonalizing the double sums below; for each fixed $m$, these sums are
finite:
\begin{equation}
\begin{gathered}
\Theta(2m) = \sum_k f(2m - a_k) = \sum_i f(2m-2b_i) +
\sum_j f(2m-2c_j-1) \\ 
= \sum_i \sum_u f(m-b_i-b_u) + \sum_j \sum_v f(m-c_j-1-c_v) = \\
\sum_i f(m-2b_i) + 2\sum_{i < u} f(m-b_i-b_u)  \\
+ \sum_j f(m-2c_j-1)+ 2\sum_{j < v} f(m-c_j-c_v-1) \equiv \Theta(m) \pmod{2}.
\end{gathered}
\end{equation}
Similarly,
\begin{equation}
\begin{gathered}
\Theta(2m+1) = \sum_k f(2m+1 - a_k) = \sum_i f(2m+1-2b_i)
+ \sum_j f(2m-2c_j) \\ 
= \sum_i \sum_j f(m-b_i-c_j) + \sum_j \sum_i
f(m-c_j-b_i) = \\
2\sum_{i,j} f(m-b_i-c_j) \equiv 0 \pmod{2}.
\end{gathered}
\end{equation}
Since $\Theta(2m) \equiv \Theta(m)$ and 
$\Theta(2m+1) \equiv 0$, $\Theta(m) \equiv 0$ for $m \ge 1$ by
induction.
\end{proof}

We now give two proofs of Theorem \ref{MainTheorem}. The first uses
Theorem \ref{rec} and the second uses the generating function
\eqref{E:genfun}. 
\begin{proof}[First proof of Theorem \ref{MainTheorem}]
Write out the product in \eqref{E:main} and use Theorem \ref{rec}.
\begin{equation}\label{E:yes}
F_{\mathcal A, 2}(x) \phi_{\mathcal A}(x) = \left(\sum_{n=0}^\infty f(n)x^n\right)
\left(1 + \sum_{i \ge 1} x^{a_i} \right) = \sum_{n=0}^\infty \Theta(n)x^n \equiv 1.
\end{equation}
\end{proof}
\begin{proof}[Second proof of Theorem \ref{MainTheorem}]
By repeated use of \eqref{E:infprod} and \eqref{E:froshdream},
\begin{equation}
\phi_{\mathcal A}(x)F_{\mathcal A,2}^2(x) \equiv 
\phi_{\mathcal  A}(x)F_{\mathcal A,2}(x^2) = \phi_{\mathcal  A}(x) 
\prod_{k=0}^\infty \phi_{\mathcal A}(x^{2^{k+1}}) = F_{\mathcal A,2}(x).
 \end{equation}
\end{proof}

The second proof generalizes immediately to primes $p > 2$ as a result
of \eqref{E:froshdream}.  
\begin{theorem}\label{genp}
If $b=p$ is prime, then 
$F_{\mathcal A, p}^{p-1}(x) \phi_{\mathcal A}(x) = 1 \in \mathbb F_p[x]$.
\end{theorem}
\begin{proof}
As before, we have
\begin{equation}
\phi_{\mathcal A}(x)F_{\mathcal A,p}^p(x) = 
\phi_{\mathcal  A}(x)F_{\mathcal A,p}(x^p) = \phi_{\mathcal  A}(x) 
\prod_{k=0}^\infty \phi_{\mathcal A}(x^{p^{k+1}}) = F_{\mathcal A,p}(x).
 \end{equation}
\end{proof}
This result  may  fail if $b$ is not prime. For
example, if $\mathcal A = \{0,1\}$ and $b=4$, then $\phi_A(x) = 1+x$
and the coefficient of $x^2$ in $F_{\mathcal A, 4}^3(x) \phi_{\mathcal A}(x)$ is $6
\not\equiv 0 \mod 4$. Note also that Theorem \ref{genp} implies that
$F_{\mathcal A, 
  p}(x) = \phi_{\mathcal A}^{-1/(p-1)}(x) \in \mathbb F_p[[x]]$. 

\begin{proof}[Proof of Corollary \ref{MainCor}(1)]
Suppose $\mathcal A$ is finite and $T$ is the period of
$\phi_{\mathcal A} (x)$.  Then by
 Theorem \ref{MainTheorem} and Lemma \ref{divide}, we have in $\mathbb F_2[x]$
\begin{equation}\label{E:comp}
F_{\mathcal A, 2}(x) = \frac 1{\phi_{\mathcal A}(x)} =
\frac{q(x)}{1+x^T},
\end{equation}
where $(1+x^T)F_{\mathcal A, 2}(x)= q(x) = 1 + \sum x^{b_k}$ and
$\deg(q) < T$. Since the coefficient of $x^{n+T}$ in
$q$ is $f(n+T) - f(n)=0$,  $(f(n) \pmod{2})$ is periodic with period $T$.
\end{proof}

Let $\mathcal A' = \{0 = b_0< b_1 < ...\}$ denote the (finite) set of
exponents which occur in $q$ in \eqref{E:comp}; $q(x) =
\phi_{\mathcal A'}(x)$. It follows from Theorem \ref{MainTheorem} that 
\begin{equation}\label{E:comp2}
F_{\mathcal A',2}(x) = \frac 1{\phi_{\mathcal A'}(x)} = \frac 1{q(x)}
= \frac{\phi_{\mathcal A}(x)}{1+x^T}.
\end{equation}
Equation \eqref{E:compleodd}  now follows from \eqref{E:comp} and
\eqref{E:comp2}. One might hope that 
$(\mathcal A')' = \mathcal A$, but that will not be the case if
$\mathcal A'$ has a smaller period than $\mathcal A$. For example, if
$\mathcal A_d =  \{0,\dots,d-1\}$, then $\phi_{\mathcal A_d}(x)(1+x) =
1+x^d$, so for each $d$, $\mathcal A_d' = \{0,1\}$. In terms of
\eqref{E:compleodd}, $f_{\mathcal A_d,2}(n)$ is odd if and only if $n
\equiv 0,1 \pmod{d}$ (as proved in \cite{Re1}) and 
$f_{\mathcal  A_d',2}(n)$ is odd if and only if $n \equiv 0,1,\dots,d-1
\pmod{d}$. That is, $f_{\mathcal  A_d',2}(n)$ is odd for all $n \ge 0$,
which is true, because it always equals 1.

  Since $(f_{\mathcal A,2}(n)  \pmod{2})$ is periodic, it is natural to ask
 for the proportion of even and odd values. It follows immediately
 from \eqref{E:compleodd} that the density of $n$ for which $f_{\mathcal
   A}(n)$ is odd is equal to $|\mathcal A'|/T$.   
 Computations with small examples lead to the conjecture that $|
\mathcal  A '| \le \frac{T+1}2$. This conjecture is false. The smallest
such example we have found is $\mathcal A_0 = \{0,1,5,9,10\}$. It turns
 out that the period of
 $\mathcal A_0$ is $33$ and $|\mathcal A_0'| = 18 >  \frac{33+1}2$. On
 the other hand, it is not hard to 
 show that if $\phi_{\mathcal A}$ is primitive, then $|\mathcal A'|
 =\frac{T+1}2 $. We  hope to say more about this topic in a future publication.

\begin{proof}[Proof of Corollary \ref{MainCor}(2)]
By Lemma \ref{periodicity}, if $\mathcal A$ is infinite, then
$(f_{\mathcal A,2}(n) \pmod{2})$ is eventually periodic if and only if
$F_{\mathcal A,2}(x)$ is a rational function, and by Theorem
\ref{MainTheorem}, this is so if and only if $\phi_{\mathcal A}(x)$
is a rational function. Suppose
\begin{equation}\label{E:end}
\phi_{\mathcal A}(x) = a(x) + \frac{q(x)}{1+x^T} \in \mathbb F_2(x), 
\end{equation}
where $a, q \in \mathbb F_2[x]$, $\deg a < N$ and $\deg q < T$ and
$q(x) = 1 + \sum_i x^{b_i}$. Then for $m > N$, $m \in \mathcal A$ if
and only if $x^m$ appears in $\phi_{\mathcal A}(x)$. By
\eqref{E:end}, this holds if and only if there exists $b_i \in
\mathcal A'$ so that $N \equiv b_i \pmod{T}$.
\end{proof}

We conclude this section with the proofs of Theorem \ref{neg}(1) and (2). 
\begin{proof}[Proof of  Theorem \ref{neg}(1)]
Let $f(n):= f_{\{0,1,2\},2}(n)$ and suppose $f(n+T) \equiv f(n) \pmod{d}$ for all
sufficiently large $n$, where $T$ is minimal.  By \eqref{E:recur2}, 
\begin{equation}
f(2m) = f(m) + f(m-1); \quad f(2m+1) = f(m)
\end{equation}
for all $m$. If $T = 2k$ is even, then for all sufficiently large $m$,
\begin{equation}
\begin{gathered}
f(2m+2k+1) \equiv f(2m+1) \pmod{d}  \implies 
f(m+k) \equiv f(m) \pmod{d},
\end{gathered}
\end{equation}
violating the minimality of $T$. 

If $T = 2k+1$ is odd, then for all sufficiently large $m$, 
\begin{equation}
\begin{gathered}
f(2m+2k+2) \equiv f(2m+1) \pmod{d} \implies \\
f(m+k+1) + f(m+k) \equiv f(m) \pmod{d},
\end{gathered}
\end{equation}
and 
\begin{equation}
\begin{gathered}
f(2m+2k+3) \equiv f(2m+2) \pmod{d} \implies \\
f(m+k+1) \equiv f(m) + f(m+1) \pmod{d}.
\end{gathered}
\end{equation}
Together, these imply that for all sufficiently large $m$, 
\begin{equation}
f(m+k) \equiv - f(m+1) \pmod{d},
\end{equation}
which implies that $f$ has a period of $2k-2$. If $k > 1$, then $0 <
2k-2 < 2k+1$ 
gives a contradiction. If $k=1$, then $T=3$. We now show that $d=2$.
First, $f(2^r-1) = f(2^{r-1}-1)$ and so by induction, $f(2^r-1) = f(1)
= 1$. Thus,  
$f(2^r) = f(2^{r-1}) + f(2^{r-1}-1) =    f(2^{r-1}) + 1$ and so by
induction, $f(2^r) = r+1$, implying that $f(2^r + 1) = f(2^{r-1}) = r$ and
$f(2^r+2) = f(2^{r-1}) + f(2^{r-1}+1) = r+r-1$. Thus, $d$ divides
$f(2^r+2) - f(2^r - 1) = 2r -1 -1$ for all sufficiently large $r$,
hence $d=2$.  
\end{proof}

\begin{proof}[Proof of  Theorem \ref{neg}(2)]
Suppose $\mathcal A = \{0,1\}$ and $b \ge 3$.  Then $f(n):= f_{\mathcal
  A,b}(n) = 1$ if $n$ is a sum of distinct powers of $b$,
and 0 otherwise. Suppose that for $n > U$, 
\begin{equation}\label{E:alleged}
f(n+T) \equiv f(n) \pmod{d}.
\end{equation}
and $d > 1$. Then, $f(m) \in \{0,1\}$ implies that $f(n+T) = f(n)$. 
Now pick $j$ so large that
$b^j > T, U$ and suppose that $f$
satisfies \eqref{E:alleged}. Then $f(b^j) = 1$, hence $f(b^j+T) = 1$,
and so $T = \sum_k b^{r_k}$ with distinct $r_k < j$.  But then $f(b^j
+ 2T) = 1$ by periodicity,
 and so $b^j + 2\sum_k b^{r_k}$ must be also a sum of distinct
powers of $b$, violating the uniqueness of the  (standard) base-$b$
representation. 
\end{proof}

\section{Examples}

\begin{example}\label{013}
The periodicity of $f_{\mathcal A_d,2}(n)$ was
established in \cite{Re1}, motivated by the interpretation of the
Stern sequence. In her dissertation, the second author \cite{D}
studied a variation on the Stern sequence defined by flipping the
recurrence \eqref{E:ste} to a two-parameter family of sequences.
 The periodicities discovered in \cite{D} for $\mathcal A = \{0,1,3\}$
 and  $\mathcal A = \{0,2,3\}$ led the third author to suggest
 studying generalizations as the topic for the 2010 summer research project
 \cite{AW} of the first and fourth authors.  

For $\al, \be \in \mathbb C$, define $b_{\al,\be}(n)$ by
\begin{equation}
\begin{gathered}
b_{\al,\be}(1) = \al, \qquad  b_{\al,\be}(2) = \be, \\
b_{\al,\be}(2n) = b_{\al,\be}(n) +b_{\al,\be}(n+1) \ \text{for $n \ge 2$}, \quad
b_{\al,\be}(2n+1) = b_{\al,\be}(n) \ \text{for $n \ge 1$}.
\end{gathered}
\end{equation}
(In order for the recurrence to be unambiguous, it can only apply
starting at $n=3$; the value of $b_{\al,\be}(0)$ plays no role.)
It is proved in \cite{D} that $b_{0,1}(n+2) = f_{\{0,2,3\},2}(n)$ for
$n \ge 0$. It was also proved there by an argument  similar to the
proof of Theorem \ref{rec} that $b_{0,1}(n) \equiv b_{0,1}(n+7)\mod 2$, and is odd
when $n \equiv 0,2,3,4 \pmod{7}$. This suggested looking at 
$f_{\{0,1,3\},2}(n)$, which is also periodic with period 7, and is odd
when $n \equiv 0,1,2,4 \pmod{7}$. The proofs of these facts are now
straightforward in view of Theorem \ref{MainTheorem}: we have in
$\mathbb F_2(x)$: 
\begin{equation}\label{E:md}
\begin{gathered}
F_{\{0,2,3\}}(x) = \frac 1{1+x^2+x^3} =
\frac{(1+x+x^3)(1+x)}{1+x^7} = \frac {1+x^2+x^3+x^4}{1+x^7}; \\
F_{\{0,1,3\}}(x) = \frac 1{1+x+x^3} =
\frac{(1+x^2+x^3)(1+x)}{1+x^7} = \frac {1+x+x^2+x^4}{1+x^7}. 
\end{gathered}
\end{equation}
Thus, $\{0,2,3\}' = \{0,2,3,4\}$ and  $\{0,1,3\}' =
\{0,1,2,4\}$. 
\end{example}

\begin{example}
For $r \ge 2$, let $\mathcal A_r =
\{0,1,2,\dots,2^r\}$ and $\mathcal B_r = \{0,1,3,\dots,2^r - 1\}$ and
let $g_r = \phi_{\mathcal A_r}$ and  $h_r = \phi_{\mathcal B_r}$ for
short. Then $g_r(x) = 1 + x h_r(x)$, so in $\mathbb F_2[x]$,
\begin{equation}
\begin{gathered}
g_r(x)h_r(x) = h_r(x) + xh_r^2 (x) = h_r(x) + xh_r(x^2) = \\
1 + \sum_{\ell=1}^r x^{2^\ell-1} + x + \sum_{\ell=1}^r
x^{2^{\ell+1}-2+1} = 1 + x^{2^{r+1}-1}.
\end{gathered}
\end{equation}
This in itself does not establish that $\mathcal A_r, \mathcal B_r$
 are complementary, or
that they both have period $2^{r+1}-1$. However, if they didn't, their
period would have to be a proper factor of $2^{r+1}-1$, which being
odd, would be at most $\frac 13 (2^{r+1}-1) < 2^r-1< 2^r$, a
contradiction. Thus $g_r$ and $h_r$ each have period $2^{r+1}-1$. We may
interpret this result combinatorially: $f_{\mathcal A_r,2}(n)$ is the
number of ways to write  
\begin{equation}
n = \sum_{i=0}^\infty \ep_i 2^{i + k_i},
\end{equation}
where $\ep_i \in \{0,1\}$ and $0 \le k_i \le r$, and $f_{\mathcal A_r,2}(n)$ 
is even, except when there exists $\ell < r$ so that
$n \equiv 2^{\ell} -1 \pmod{2^{r+1}-1}$. 
\end{example}

\begin{example}\label{0149}
We return to  $\mathcal A = \{0,1,4,9\}$; in $\mathbb F_2[x]$,
\begin{equation}\label{E:fact}
\phi_{\mathcal A}(x) = 1 + x + x^4 + x^9 = (1+x)^4(1+x+x^2)(1+x^2+x^3).
\end{equation}
Note that $1+x$ has period 1,
$1+x+x^2$ has period 3, and we have already seen that $1+x+x^3$ has
period 7. Since the maximum exponent in \eqref{E:fact} is $\le 2^2$,
Lemma \ref{divide} implies that the period of $\mathcal A$ divides $4\cdot
\text{lcm}(1,3,7) = 84$. Another calculation shows that
$\phi_{\mathcal A}(x)$ does not divide $1 + x^{\frac{84}p}$ for $p =
2,3,7$, and so 84 is actually the period. A computation shows that 
$\mathcal A' = \{0,1,2,3,\dots,70,75\}$ has 41 terms, as noted
earlier. Thus $f_{\mathcal A}(n)$ is odd $\frac{41}{84}$ of the time
and even $\frac{43}{84}$ of the time.  
\end{example}

\begin{example}
Although  Theorem \ref{MainTheorem} does not
generalize to all $\mathcal A$ if $(b,d) \neq (2,2)$,  there
are a few exceptional cases. Problem B2 on the 1983 Putnam \cite{KAH}
in effect asked for a proof that for $\mathcal A = \{0,1,2,3\}$, 
\begin{equation}
f_{\mathcal A,2}(n) = \Big\lfloor \frac n2 \Big\rfloor + 1.
\end{equation}
This can be seen directly from (1.4), since $\phi_{\mathcal A,2}(x) =
(1+x)(1+x^2) = \frac{1-x^4}{1-x}$, hence $F_{\mathcal A,2}(x)$
telescopes to $\frac 1{(1-x)(1-x^2)}$. It follows immediately that
$f_{\mathcal A,2}(n+2d)  = f_{\mathcal A,2}(n)  + d$,  and hence $f_{\mathcal A, 2}$ is
periodic mod $d$ for each $d$, with period $2d$. 
A similar phenomenon occurs for $\mathcal A_b = \{0,1,\dots,b^2-1\}$, so
that $\phi_{\mathcal A_b,b}(x) = \frac{1-x^{b^2}}{1-x}$ and $F_{\mathcal
  A_b,b}(x) = \frac 1{(1-x)(1-x^b)}$, implying that $f_{\mathcal A_b,b}(n)
= \lfloor \frac nb \rfloor + 1$ and $f_{\mathcal A_b,b}(n+bd)  =
f_{\mathcal A_b,b}(n)  + d$. 
\end{example}

\begin{example}
Let $\mathcal A = 0 \cup (2\mathbb N + 1)$ (all
non-zero digits in \eqref{E:basicrep} are odd). Then
\begin{equation}
\phi_{\mathcal A}(x) = 1 + \sum_{i=0}^\infty x^{2i+1} = 1 + \frac
x{1-x^2} = \frac{1+x-x^2}{1-x^2}. 
\end{equation}
Working in $\mathbb F_2(x)$, we have
\begin{equation}
F_{\mathcal A,2}(x) = \frac{1-x^2}{1+x-x^2} = \frac{(1+x)^2}{1+x+x^2}
= 1 + \frac{x}{1+x+x^2} = 1 + \frac{x+x^2}{1+x^3}.
\end{equation}
Thus, $f_{\mathcal A,2}(n)$ is odd if and only if $n = 0$ or $n$ is
not a multiple of 3. 
\end{example} 

\begin{example}
Let $\mathcal A^{\{k\}} := \mathbb N \setminus
\{k\}$. By Theorem \ref{MainTheorem},
\begin{equation}
\begin{gathered}
\phi_{\mathcal  A^{\{k\}}} = \frac 1{1+x} -  x^k = \frac{1 -
  x^k-x^{k+1}}{1+x} \implies  F_{\mathcal  A^{\{k\}}}(x) = \frac
{1+x}{1+x^k+x^{k+1}} \\ 
\implies  F_{\mathcal  A^{\{1\}}}(x) = \frac {1+x}{1+x+x^2} = \frac
{(1+x)^2}{1+x^3} = \frac{1 + x^2}{1+x^3}.
\end{gathered}
\end{equation}
Thus $f_{\mathbb N \setminus \{1\},2}(n)$ is odd precisely when $n \equiv 0,2
\pmod{3}$. This may be contrasted with with $f_{\{0,1,2\},2}(n)$,
which is odd precisely when $n \equiv 0,1 \pmod{3}$.
\end{example}

\section{Bibliography}


\begin{thebibliography}{23}

\bibitem{AW} K. Anders and J. Weber, \emph{The parity of generalized
    binary representations}, REGS report, Aug. 30, 2010. 



\bibitem{B} E. Berlekamp, \emph{Algebraic coding theory, Revised 1984
    edition} Aegean Park Press, Laguna Hills, CA, 1984. MR0238597 (38
  \#6873) (review of first edition).

\bibitem{BZ} R. P. Brent and P. Zimmerman, \emph{The great trinomial
    hunt}, Trans. Amer. Math. Soc., \textbf{58} (2011), 233-239.

\bibitem{C} R. F. Churchhouse, \emph{Congruence properties of the
    binary partition function}, Proc. Cambridge
  Philos. Soc., \textbf{66} (1969), 371–376, MR0248102 (40 \#1356).

\bibitem{D} M. Dennison, \emph{A sequence related to the Stern
    sequence}, Ph.D. dissertation, University of Illinois at
  Urbana-Champaign, 2010.

\bibitem{KAH} L. F. Klosinski, G. L. Alexanderson and A. P. Hillman,
  \emph{The William Lowell Putnam Mathematical Competition},
  Amer. Math. Monthly, \textbf{91} (1984), 487–495,  MR1540495. 

\bibitem{L} D. H. Lehmer, On Stern's diatomic series, {\it
  Amer. Math. Monthly},  {\bf 36} (1929), 59--67, MR1521653.



\bibitem{Re1} B. Reznick, \emph{Some binary partition functions},  Analytic
  number theory (Allerton Park, IL, 1989),  451–477, Progr. Math., 85,
  Birkhäuser Boston, Boston, MA, 1990, MR1084197 (91k:11092).

\bibitem{Re2} B. Reznick, \emph{Regularity properties of the Stern
    enumeration of the rationals},  J. Integer Seq. \textbf{11} (2008), no. 4,
  Article 08.4.1, MR2447843 (2009g:11016). 

\bibitem{RS} \O. J. R{\o}dseth and J. E. Sellers, \emph{On $m$-ary partition
    function congruences: a fresh look at a past problem}, J. Number
  Theory, {\bf 87} (2001), 270–281, MR1824148 (2001m:11177).


\bibitem{Ste} M. A. Stern, \emph{Ueber eine zahlentheoretische Funktion},
  J. Reine Angew. Math., {\bf 55} (1858) 193--220.




\end{thebibliography}
\end{document}